\documentclass[11pt]{article}
\usepackage{enumerate}
\usepackage{amssymb,a4wide,latexsym,makeidx,epsfig,fleqn}
\usepackage{amsthm}
\usepackage{amsmath}
\usepackage{enumerate}

\newtheorem{lem}{Lemma}[section]
\newtheorem{thm}{Theorem}[section]

\newtheorem{conj}{Conjecture}

\newtheorem{clm}{Claim}[section]
\allowdisplaybreaks[4]
\usepackage{graphics,multicol}
\usepackage{graphicx}
\usepackage{tikz}
\usepackage{caption}
\usepackage{booktabs}
\usepackage[noadjust]{cite}
\begin{document}
\textwidth 150mm \textheight 225mm
\title{Maximizing the spectral radius of graphs of given size with forbidden a subgraph
\thanks{Supported by the National Natural Science Foundation of China (No.  12271439).}}
\author{{Yanting Zhang$^{a,b}$, Ligong Wang$^{a,b,}$\thanks{Corresponding author.}}\\
\small $^a$ School of Mathematics and Statistics,\\
\small Northwestern Polytechnical University, Xi’an, Shaanxi 710129, P.R. China\\
\small $^b$ Xi’an-Budapest Joint Research Center for Combinatorics,\\ 
\small Northwestern Polytechnical University,  Xi’an, Shaanxi 710129, P.R. China\\
\small E-mail:
yantzhang@163.com (Y. Zhang), lgwangmath@163.com (L. Wang)}
\date{}
\maketitle
\begin{center}
\begin{minipage}{120mm}
\vskip 0.3cm
\begin{center}
{\small {\bf Abstract}}
\end{center}
{\small 
Let $H_7$ denote the $7$-vertex \textit{fan graph} consisting of a $6$-vertex path plus a vertex adjacent to each vertex of the path. Let $K_3 \vee \frac{m-3}{3}K_1$ be the graph obtained by joining each vertex of a triangle $K_3$ to $\frac{m-3}{3}$ isolated vertices. In this paper, we show that if $G$ is an $H_{7}$-free graph with size $m\geq 33$, then the spectral radius $
\rho(G)\leq 1+\sqrt{m-2},$ equality holds if and only if $G\cong K_3 \vee \frac{m-3}{3}K_1$ (possibly, with some isolated vertices).

}
\end{minipage}
\end{center}

\section{Introduction }
\label{sec:ch6-introduction}
Let $G$ be an undirected and simple graph with vertex set $V(G)$ and edge set $E(G)$. We use $m:=|E(G)|$ to denote the \textit{size} of $G$. The \textit{adjacency matrix} of $G$ is defined as $A(G)=(a_{u,v})_{u,v\in V(G)}$, where $a_{u,v}=1$ if $uv\in E(G)$, and $a_{u,v}=0$ otherwise. The largest eigenvalue of $A(G)$, denoted by $\rho(G)$, is called the (adjacency) \textit{spectral radius} of $G$.
For two vertex-disjoint graphs $G$ and $H$, $G \cup H$ denotes the \textit{union} of $G$ and $H$, and $G \vee H$ denotes the \textit{join} of $G$ and $H$, i.e., joining every vertex of $G$ to every vertex of $H$.
Let $kG$ denote the union of $k$ vertex-disjoint copies of $G$.
As usual, denote by $P_{n}$, $K_{n}$ and $K_{a, n-a}$ the \textit{path}, the \textit{complete graph} and the \textit{complete bipartite graph} of order $n$, respectively. For two subsets $R, S \subseteq V(G)$, we use $e(R, S)$ to denote the number of all edges of $G$ with one end vertex in $R$ and the other in $S$, and particularly, $e(S,S)$ is simplified by $e(S)$. 
For a vertex $v \in V(G)$, we write $N_{G}(v)$ for the set of neighbors of $v$ in $G$. Denote by $d_{G}(v):=|N_{G}(v)|$ the \textit{degree} of a vertex $v$ in $G$. Let $\delta(G)$ denote the \textit{minimum degree} of $G$. For a subset $S$ of $V(G)$, let $N_S(v):=N_{G}(v)\cap S$ and $d_S(v):=|N_S(v)|$. Moreover, we use $G[S]$ to denote the subgraph of $G$ induced by $S$. For a positive integer $k$, we use $G-ke$ to denote the graph obtained from $G$ by deleting any $k$ edges. For graph notations and concepts undefined here, readers are referred to \cite {BAJ}.

The investigation about the maximum (or minimum) spectral radius of graphs is an important and classic topic in graph spectra theory.
Indeed, the problem of characterizing graphs of given size with maximum spectral radius was initially proposed by Brualdi and Hoffman \cite {BRU} as a conjecture, and completely solved by Rowlinson \cite {Row}. 
Given a graph $H$, a graph is said to be \textit{$H$-free} if it does not contain $H$ as a subgraph. In spectral extremal graph theory, it is interesting to study the extremal problem asks what is the maximum spectral radius of an $H$-free graph of given size $m$, and characterize the corresponding extremal graphs. Up to now, the above problem has been investigated for various kinds of $H$, see for example, the complete graph \cite{NVB,NVB1,NVBF3}, the complete bipartite graph \cite{MZS},  the friendship graph \cite{Li1,Yu}, the book graph \cite{NVB2} and the cycles with specified length \cite{Lu,Min,NVBF,MZS}. 

The \textit{fan graph} is defined as $H_{t+1}:\cong K_1 \vee P_t$. It is worth noting that there have been plentiful results about the extremal problem on spectral radius for $H_{t+1}$-free graphs with size $m$.  In 1970, Nosal \cite{NEE} showed that $\rho(G) \leq \sqrt{m}$ for every $H_{3}$-free graph $G$ with size $m$. Subsequently,  Nikiforov \cite{NVB, NVB1} extended Nosal's result, and proved that if $G$ is an extremal graph with maximum spectral radius among all $H_{3}$-free graphs of size $m$, then $G$ is isomorphic to a complete bipartite graph (possibly, with some isolated vertices). In 2024, Lu, Lu and Li \cite{Lu} concluded that if $G$ is an extremal graph with maximum spectral radius over all $H_{4}$-free graphs of size $m$, then $G$ is also isomorphic to a complete bipartite graph (possibly, with some isolated vertices). Very recently, Zhang and Wang \cite{Zhang} (resp. Yu, Li and Peng \cite{Yu}) showed that if $G$ is an extremal graph with maximum spectral radius among all $H_{5}$-free graphs of size $m\geq 11$ (resp. $m\geq 8$), then $G \cong K_2 \vee \frac{m-1}{2} K_1$ (possibly, with some isolated vertices). Additionally, Yu, Li and Peng \cite{Yu} proposed the following conjecture on spectral radius for $H_{t+1}$-free graphs with given size $m$.
\begin{conj}\emph{(\cite{Yu})}\label{conj::1}
Let $k\geq2$ be a fixed positive integer. If  $G$ is an $H_{2k+1}$-free or $H_{2k+2}$-free graph of sufficiently large size $m$ without isolated vertices. Then 
$$
\rho(G)\leq \frac{k-1+\sqrt{4m-k^{2}+1}}{2},
$$
equality holds if and only if $G\cong K_{k} \vee\left(\frac{m}{k}-\frac{k-1}{2}\right) K_{1}$.
\end{conj}
In this paper, we will consider the case $k=3$ for $H_{2k+1}$-free graphs. Specifically, we attain the maximum spectral radius among all $H_7$-free graphs of size $m\geq33$, and characterize the unique extremal graph having no isolated vertices. Our main result is as follows. 

\begin{thm}\label{thm::1.1}
Let $m\geq 33$ and $G$ be an $H_{7}$-free graph with size $m$ having no isolated vertices. Then 
$$
\rho(G)\leq 1+\sqrt{m-2},
$$
equality holds if and only if $G\cong K_3 \vee \frac{m-3}{3}K_1$.
\end{thm}

Let $C_6^{\triangle}$ denote the graph consisting of a $6$-vertex cycle plus a vertex adjacent to two adjacent vertices of the cycle. In 2024, Lu, Lu and Li \cite{Lu} considered the extremal problem on spectral radius for $C_6^{\triangle}$-free graphs with size $m \geq38$.
\begin{thm}\emph{(\cite{Lu})}\label{thm::1.2}
Let $m\geq 38$ and $G$ be a $C_6^{\triangle}$-free graph with size $m$ having no isolated vertices. Then 
$$
\rho(G)\leq 1+\sqrt{m-2},
$$
equality holds if and only if $G\cong K_3 \vee \frac{m-3}{3}K_1$.
\end{thm} 

Let $F_3$ denote the graph obtained from three vertex-disjoint triangles by sharing a common vertex. Very recently, Yu, Li and Peng \cite{Yu} 
studied the extremal problem on spectral radius for $F_3$-free graphs with size $m \geq33$.
\begin{thm}\emph{(\cite{Yu})}\label{thm::1.3}
Let $m\geq 33$ and $G$ be an $F_{3}$-free graph with size $m$ having no isolated vertices. Then 
$$
\rho(G)\leq 1+\sqrt{m-2},
$$
equality holds if and only if $G\cong K_3 \vee \frac{m-3}{3}K_1$.
\end{thm} 
Note that $C_6^{\triangle}$ and $ F_{3}$ are two spanning subgraphs of $H_7$. Hence, Theorem \ref{thm::1.1} can imply Theorems \ref{thm::1.2} and \ref{thm::1.3}.

\section{Proof of Theorem \ref{thm::1.1}}
We first present some important results that are useful for our arguments. 

\begin{lem}\emph{(\cite{BBR})}\label{lem::2.1}
Let $G$ be a connected graph. If $H$ is a proper subgraph of $G$, then 
$\rho(H)<\rho(G).$
\end{lem}

\begin{lem}\emph{(\cite{Wu})}\label{lem::2.2}
Let $u, v$ be two distinct vertices in a connected graph $G$. Suppose that $v_1, v_2, \dots, v_s$ $(1\leq s \leq d_G(v))$ are some vertices of $ N_G(v) \backslash N_G(u)$, and 
$\mathbf{x}$ is the Perron vector of $A(G)$ with coordinate $x_z$ corresponding to $z\in V(G)$. Let $G^{\prime}= G-\{v v_i \mid 1 \leq i \leq s\}+\{u v_i \mid 1 \leq i \leq s\}$. If $x_u \geq x_v$, then $\rho(G)<\rho(G^{\prime}).$
\end{lem}

Let $\lambda(M)$ denote the largest eigenvalue of a square matrix $M$ with only real eigenvalues.
\begin{lem}(p. 30, \cite{Brou}; p. 196--198, \cite{Gods})\label{lem::2.5}
Let $M$ be a real symmetric matrix and $B_{\Pi}$ be an equitable quotient matrix of $M$. Then the eigenvalues of $B_{\Pi}$ are also eigenvalues of $M$. Furthermore, if $M$ is nonnegative and irreducible, then
$\lambda(M)=\lambda\left(B_{\Pi}\right).$
\end{lem}

\begin{lem}\label{lem::2.4}
Let $m\geq6$. Then $\rho(K_3 \vee \frac{m-3}{3}K_1) = 1+\sqrt{m-2}.$
\end{lem}
\begin{proof}
Let $G:\cong K_3 \vee \frac{m-3}{3}K_1$. It is easy to see that $A(G)$ has the equitable quotient matrix
$$
B_{\Pi}=\left[\begin{array}{cc}
2 & \frac{m-3}{3}\\
3 & 0
\end{array}\right].
$$
Thus, the characteristic polynomial of $B_{\Pi}$ is $f(x)=x^2-2x-m+3.$
By Lemma \ref{lem::2.5}, $\rho(G)=\lambda\left(B_{\Pi}\right)$ is the largest root of $f(x)=0$. Hence, $\rho(G) = 1+\sqrt{m-2}.$
\end{proof}		

The following result can be obtained by the proofs of Lemmas $4$ and $5$ in \cite{Lu}. 
\begin{lem}\emph{(\cite{Lu})}\label{lem::2.3}
Let $G$ be a connected graph, $\mathbf{x}$ be the Perron vector of $A(G)$ with coordinate $x_u$ corresponding to $u\in V(G)$ and $u^* \in V(G)$ with $x_{u^*}=\max \{x_u: u \in V(G)\}$.  Suppose that $H$ is any non-trivial component of $G[N_G(u^*)]$ and $\eta(H):=\sum_{u \in V(H)}(d_H(u)-2) \frac{x_u}{x_{u^*}}-e(H)$. If $G[N_G(u^*)]$ is $P_6$-free, then the following statements hold.

\emph{(i)} For $\delta(H) \geq 2$, 
$$
\eta(H) \leq \begin{cases}0 & \text { if } H\cong K_5, \\ -1 & \text { if } H\cong K_5-e, \\ -2 & \text { if } H\cong K_4 \text { or } K_5-2 e, \\ -3 & \text { otherwise, }\end{cases}
$$
equality holds only if $x_u=x_{u^*}$ for all $u \in V(H) \backslash\{v \in V(H): d_H(v)=2\}$. Particularly, if $|V(H)|\geq6$, then $\eta(H)=-3$ occurs only if $H\cong K_2 \vee (|V(H)|-2)K_1$. 

\emph{(ii)} For $\delta(H) = 1$, 
$$
\eta(H) < \begin{cases} -1 & \text { if } H\cong K_2, \\ -2 & \text { if } H\cong K_{1,s} \;(s\geq 2) \text { or } K_1 \vee (K_3 \cup (|V(H)|-4)K_1) \;(|V(H)|\geq5), \\ -3 & \text { otherwise. }\end{cases}
$$
\end{lem}
In the sequel, we will give the proof of Theorem \ref{thm::1.1}.
\renewcommand\proofname{\bf Proof of Theorem \ref{thm::1.1}.}
\begin{proof}
\renewcommand\proofname{\bf Proof.}	
Let $m\geq33$ and $G$ be an extremal graph with the maximum spectral radius among all $H_7$-free graphs with size $m$ having no isolated vertices. Set $\rho:=\rho(G)$ for short. 
Indeed, $G$ is connected. Otherwise, let $G_1$ be a component of $G$ with $\rho(G_1)=\rho$ and $G_2$ be any other component. Assume that $u \in V(G_1)$  and $vw \in E(G_2)$. Let $G^{\prime}$ be the graph obtained from $G$ by adding an edge $uv$ and deleting the edge $vw$. Note that $uv$ is a cut edge in $G^{\prime}$ and $H_7$ is $2$-connected. Then $G^{\prime}$ is still an $H_7$-free graph with size $m$. However, by Lemma \ref{lem::2.1} we obtain $\rho(G^{\prime}) > \rho$, a contradiction. 
Suppose that $\mathbf{x}$ is the Perron vector of $A(G)$ with  coordinate $x_u$ corresponding to $u\in V(G)$, and $u^*$ is a vertex of $V(G)$ with $$x_{u^*}=\max\{x_{u} \mid u \in  V(G)\}.$$
Let $A=N_G(u^*)$ and $B=V(G)\setminus(A\cup \{u^*\})$. Moreover, we partition $A$ into two subsets $A_0$ and $A_+$, where $A_0=\{u\in A \mid d_A(u)=0\}$ and $A_+=A\setminus A_0$. It is clear that 
\begin{align}\label{equ::3.12}
m=|A|+e(A_+)+e(A,B)+e(B).
\end{align}
Since $A(G)\mathbf{x}=\rho\mathbf{x}$, we have
\begin{align}\label{equ::3.9}
\rho x_{u^*}=\sum_{u \in A} x_u=\sum_{u \in A_+} x_u+\sum_{v \in A_0} x_v.
\end{align}
Additionally, we obtain $A^2(G)\mathbf{x}=\rho^2\mathbf{x}$, and hence
\begin{align}\label{equ::3.2}
\rho^2x_{u^*}=\sum_{v \in V(G)} a^{(2)}_{u^*v}x_v=|A|x_{u^*}+\sum_{u \in A_+}d_{A}(u) x_u+\sum_{w \in B}d_{A}(w) x_w,
\end{align}
where $a^{(2)}_{u^*v}$ denotes the number of walks of length $2$ starting from $u^*$ to $v$.
Since $K_3 \vee \frac{m-3}{3}K_1$ is an $H_7$-free graph with size $m$, by Lemma \ref{lem::2.4} we have
\begin{align}\label{equ::3.3}
\rho\geq\rho( K_3 \vee \frac{m-3}{3}K_1) = 1+\sqrt{m-2},
\end{align}
which deduces $\rho^2-2\rho\geq m-3$. Owing to $m\geq33$, we obtain 
\begin{align}\label{equ::3.1}
\rho\geq1+\sqrt{m-2}\geq 1+\sqrt{31}.
\end{align}
It follows from \eqref{equ::3.12}-\eqref{equ::3.2} that
\begin{align*}
(m-3)x_{u^*}\leq(\rho^2-2\rho)x_{u^*}
&=|A|x_{u^*}+\sum_{u \in A_+}(d_{A}(u)-2) x_u-2\sum_{v \in A_0} x_v+\sum_{w \in B}d_{A}(w) x_w\\
&\leq|A|x_{u^*}+\sum_{u \in A_+}(d_{A}(u)-2) x_u-2\sum_{v \in A_0} x_v+e(A,B))x_{u^*}\\
&=(m-e(B)-e(A_+)+ \sum_{u \in A_+}(d_{A}(u)-2)\frac{x_u}{x_{u^*}} -2\sum_{v \in A_0} \frac{x_v}{x_{u^*}})x_{u^*},
\end{align*}
which indicates
\begin{align}\label{equ::3.4}
e(B) \leq \sum_{u \in A_+}(d_{A}(u)-2)\frac{x_u}{x_{u^*}}-e(A_+)-2\sum_{v \in A_0} \frac{x_v}{x_{u^*}}+3,
\end{align}
equality holds if and only if  $\rho^2-2\rho= m-3$ and $x_{w}=x_{u^{*}}$ for any $ w \in B$ satisfying $d_{A}(w) \geq 1$. Denote by $\varGamma$ the set of all components of $G[A_+]$.
For each component $H\in\varGamma$, we denote
$$
\eta(H):=\sum_{u \in V(H)}(d_{H}(u)-2) \frac{x_u}{x_{u^*}}-e(H).
$$
Then by \eqref{equ::3.4} we obtain
\begin{align}\label{equ::3.7}
e(B) \leq \sum_{H \in \varGamma} \eta(H)-2\sum_{v \in A_0} \frac{x_v}{x_{u^*}}+3.
\end{align}
Since $G$ is $H_7$-free, we know that $G[A]$ is $P_6$-free. Combining this with Lemma \ref{lem::2.3}, we can get the following claim.
\begin{clm}\label{clm::2.1}
For each component $H\in\varGamma$, the following statements hold.

\emph{(i)} For $\delta(H) \geq 2$, 
$$
\eta(H) \leq \begin{cases}0 & \text { if } H\cong K_5, \\ -1 & \text { if } H\cong K_5-e, \\ -2 & \text { if } H\cong K_4 \text { or } K_5-2 e, \\ -3 & \text { otherwise, }\end{cases}
$$
equality holds only if $x_u=x_{u^*}$ for all $u \in V(H)$ with $d_H(u)\geq3$. Particularly, if $|V(H)|\geq6$, then $\eta(H)=-3$ occurs only if $H\cong K_2 \vee (|V(H)|-2)K_1$. 

\emph{(ii)} For $\delta(H) = 1$, 
$$
\eta(H) < \begin{cases} -1 & \text { if } H\cong K_2, \\ -2 & \text { if } H\cong K_{1,s} \;(s\geq 2) \text { or } K_1 \vee (K_3 \cup (|V(H)|-4)K_1) \;(|V(H)|\geq5), \\ -3 & \text { otherwise. }\end{cases}
$$
\end{clm}
According to Claim \ref{clm::2.1}, we find that $\eta(H)\leq 0$ for each component $H \in \varGamma$. Together with \eqref{equ::3.7}, we obtain
\begin{align}\label{equ::3.5}
e(B) \leq 3.
\end{align} 
\begin{clm}\label{clm::2.2}
$H \ncong K_5$  for any component $H \in \varGamma$.
\end{clm}
\begin{proof}
Suppose to the contrary that $H$ is a component in $\varGamma$ such that $H\cong K_5$. Let $V(H)=\{u_1,u_2,u_3,u_4,u_5\}$ with $x_{u_5}\geq x_{u_4}\geq x_{u_3}\geq x_{u_2}\geq x_{u_1}$. Clearly, $d_H(u_i)=4$ for each $i \in \{1,2,3,4,5\}$. We will distinguish two cases to lead a contradiction, respectively.

{\flushleft\bf Case 1.} \textbf{$\bigcup\limits_{i=1}^5N_B(u_i)=\emptyset$.}

In this case, we have $x_{u_1}=x_{u_2}=x_{u_3}=x_{u_4}=x_{u_5}$ and $\rho x_{u_1}=x_{u_2}+x_{u_3}+x_{u_4}+x_{u_5}+x_{u^*}$. Together with  \eqref{equ::3.1}, we obtain  $x_{u_1}=\frac{1}{\rho-4}x_{u^*}<\frac{2}{5}x_{u^*}$. Thus, by \eqref{equ::3.7} we get
\begin{align*}
e(B) &\leq \eta(H) + \sum_{H^{\prime} \in \varGamma \setminus \{H\}} \eta(H^{\prime})-2\sum_{v \in A_0} \frac{x_v}{x_{u^*}}+3\\
&=\sum_{i=1}^5(d_{H}(u_i)-2) \frac{x_{u_i}}{x_{u^*}}-e(H)+ \sum_{H^{\prime} \in \varGamma \setminus \{H\}} \eta(H^{\prime})-2\sum_{v \in A_0} \frac{x_v}{x_{u^*}}+3\\
&=10\cdot\frac{x_{u_1}}{x_{u^*}}-10+ \sum_{H^{\prime} \in \varGamma \setminus \{H\}} \eta(H^{\prime})-2\sum_{v \in A_0} \frac{x_v}{x_{u^*}}+3\\
&<-3,
\end{align*}
a contradiction.

{\flushleft\bf Case 2.} \textbf{$\bigcup\limits_{i=1}^5N_B(u_i)\neq\emptyset$.}

In this case, let $w$ be an arbitrary vertex in $\bigcup_{i=1}^5N_B(u_i)$. Since $G$ is $H_7$-free, we see that $|N_{G}(w) \cap \{u_1, u_2, u_3,u_4,u_5\}|=1$. Indeed, $N_{G}(w) \cap \{u_1, u_2, u_3,u_4,u_5\}=\{u_5\}$. Otherwise, we may assume without loss of generality that $u_1$ is the unique vertex in $N_{G}(w) \cap \{u_1, u_2, u_3,u_4,u_5\}$. Let $G^{\prime}= G-wu_1+wu_5$. Recall that $e(B) \leq3$ from \eqref{equ::3.5}. Then one can observe that $G^{\prime}$ is still an $H_{7}$-free graph with size $m$. Owing to $x_{u_5}\geq x_{u_1}$, by Lemma \ref{lem::2.2} we obtain $\rho(G^{\prime})>\rho$, a contradiction. Thus, $N_B(u_i)=\emptyset$ for each $i\in \{1,2,3,4\}$, which indicates that $x_{u_1}=x_{u_2}=x_{u_3}=x_{u_4}$ and $\rho x_{u_1}=x_{u_2}+x_{u_3}+x_{u_4}+x_{u_5}+x_{u^*}$. Together with \eqref{equ::3.1}, we have $x_{u_1}\leq\frac{2}{\rho-3}x_{u^*}<\frac{5}{8}x_{u^*}$. Hence, by \eqref{equ::3.7} we obtain
\begin{align*}
e(B) &\leq \eta(H) + \sum_{H^{\prime} \in \varGamma \setminus \{H\}} \eta(H^{\prime})-2\sum_{v \in A_0} \frac{x_v}{x_{u^*}}+3\\
&=\sum_{i=1}^5(d_{H}(u_i)-2) \frac{x_{u_i}}{x_{u^*}}-e(H)+ \sum_{H^{\prime} \in \varGamma \setminus \{H\}} \eta(H^{\prime})-2\sum_{v \in A_0} \frac{x_v}{x_{u^*}}+3\\
&=8\cdot\frac{x_{u_1}}{x_{u^*}}+2\cdot\frac{x_{u_5}}{x_{u^*}}-10+ \sum_{H^{\prime} \in \varGamma \setminus \{H\}} \eta(H^{\prime})-2\sum_{v \in A_0} \frac{x_v}{x_{u^*}}+3\\
&<0,
\end{align*}
which is also a contradiction.
\end{proof}	
\begin{clm}\label{clm::2.4}
$e(A_+) \geq 11$.
\end{clm}
\begin{proof}
Suppose to the contrary that $e(A_+) \leq 10$. In view of \eqref{equ::3.12} and \eqref{equ::3.2}, we obtain
\begin{align*}
\rho^2x_{u^*}&=|A|x_{u^*}+\sum_{u \in A_+}d_{A}(u) x_u+\sum_{w \in B}d_{A}(w) x_w\\
&\leq (|A|+2e(A_+)+e(A,B))x_{u^*}=(m-e(B)+e(A_+))x_{u^*}\leq(m+10)x_{u^*},
\end{align*}
and so $\rho\leq \sqrt{m+10} <1+\sqrt{m-2}$ due to $m\geq33$,  which contradicts \eqref{equ::3.3}.
\end{proof}		
\begin{clm}\label{clm::2.3}
$H \ncong K_5-e$  for any component $H \in \varGamma$.
\end{clm}
\begin{proof}
Suppose to the contrary that $H$ is a component in $\varGamma$ such that $H\cong K_5-e$. Let $V(H)=\{u_1,u_2,u_3,u_4,u_5\}$ with $d_H(u_1)=d_H(u_2)=d_H(u_3)=4$, $d_H(u_4)=d_H(u_5)=3$ and $x_{u_3}\geq x_{u_2}\geq x_{u_1}$. We will distinguish two cases to lead a contradiction, respectively.

{\flushleft\bf Case 1.} \textbf{$\bigcup\limits_{i=1}^3N_B(u_i)=\emptyset$.}

In this case, we have $x_{u_1}=x_{u_2}=x_{u_3}$ and $\rho x_{u_1}=x_{u_2}+x_{u_3}+x_{u_4}+x_{u_5}+x_{u^*}$. Together with \eqref{equ::3.1}, we obtain  $x_{u_1}\leq\frac{3}{\rho-2}x_{u^*}<\frac{2}{3}x_{u^*}$.  Thus, by \eqref{equ::3.7} we get
\begin{align*}
e(B) &\leq \eta(H) + \sum_{H^{\prime} \in \varGamma \setminus \{H\}} \eta(H^{\prime})-2\sum_{v \in A_0} \frac{x_v}{x_{u^*}}+3\\
&=\sum_{i=1}^5(d_{H}(u_i)-2) \frac{x_{u_i}}{x_{u^*}}-e(H)+ \sum_{H^{\prime} \in \varGamma \setminus \{H\}} \eta(H^{\prime})-2\sum_{v \in A_0} \frac{x_v}{x_{u^*}}+3\\
&=6\cdot\frac{x_{u_1}}{x_{u^*}}+\frac{x_{u_4}}{x_{u^*}}+\frac{x_{u_5}}{x_{u^*}}-9+ \sum_{H^{\prime} \in \varGamma \setminus \{H\}} \eta(H^{\prime})-2\sum_{v \in A_0} \frac{x_v}{x_{u^*}}+3\\
&<0,
\end{align*}
a contradiction. 

{\flushleft\bf Case 2.} \textbf{$\bigcup\limits_{i=1}^3N_B(u_i)\neq\emptyset$.}

In this case, let $w$ be an arbitrary vertex in $ \bigcup_{i=1}^3N_B(u_i)$. Since $G$ is $H_7$-free, we see that $|N_{G}(w) \cap \{u_1, u_2, u_3\}|=1$ and $N_{G}(w) \cap \{u_4, u_5\}=\emptyset$. Indeed, $N_{G}(w) \cap \{u_1, u_2, u_3\}=\{u_3\}$. Otherwise, we may assume without loss of generality that $u_1$ is the unique vertex in $N_{G}(w) \cap \{u_1, u_2, u_3\}$. Let $G^{\prime}= G-wu_1+wu_3$. Since $e(B) \leq3$ from \eqref{equ::3.5}, we can find that $G^{\prime}$ is also an $H_{7}$-free graph with size $m$. Due to $x_{u_3}\geq x_{u_1}$, by Lemma \ref{lem::2.2} we obtain $\rho(G^{\prime})>\rho$, a contradiction. Hence, $N_B(u_1)=N_B(u_2)=\emptyset$, which implies that $x_{u_1}=x_{u_2}$ and $\rho x_{u_1}=x_{u_2}+x_{u_3}+x_{u_4}+x_{u_5}+x_{u^*}$. Together with \eqref{equ::3.1}, we obtain  
\begin{align}\label{equ::3.10}
x_{u_1}\leq\frac{4}{\rho-1}x_{u^*}<\frac{4}{5}x_{u^*}.
\end{align}
Notice that $e(A_+)\geq 11$ from Claim \ref{clm::2.4}. Then there exists a component $H^{\prime}\in\varGamma \setminus \{H\}$. If $H^{\prime}\cong K_5-e$, then let $V(H^{\prime})=\{u^{\prime}_1,u^{\prime}_2,u^{\prime}_3,u^{\prime}_4,u^{\prime}_5\}$ with $d_{H^{\prime}}(u^{\prime}_1)=d_{H^{\prime}}(u^{\prime}_2)=d_{H^{\prime}}(u^{\prime}_3)=4$, $d_{H^{\prime}}(u^{\prime}_4)=d_{H^{\prime}}(u^{\prime}_5)=3$ and $x_{u^{\prime}_3}\geq x_{u^{\prime}_2}\geq x_{u^{\prime}_1}$. By the above discussion, we get $x_{u^{\prime}_1}=x_{u^{\prime}_2}<\frac{4}{5}x_{u^*}$.
Therefore, by \eqref{equ::3.7} we have
\begin{align*}
e(B) \leq& \eta(H)+\eta(H^{\prime}) + \sum_{H^{\prime\prime} \in \varGamma \setminus \{H,H^{\prime}\}} \eta(H^{\prime\prime})-2\sum_{v \in A_0} \frac{x_v}{x_{u^*}}+3\\
=&\sum_{i=1}^5(d_{H}(u_i)-2) \frac{x_{u_i}}{x_{u^*}}-e(H)+\sum_{j=1}^5(d_{H^{\prime}}(u^{\prime}_j)-2) \frac{x_{u^{\prime}_j}}{x_{u^*}}-e(H^{\prime})+\\
& \sum_{H^{\prime\prime} \in \varGamma \setminus \{H,H^{\prime}\}} \eta(H^{\prime\prime})-2\sum_{v \in A_0} \frac{x_v}{x_{u^*}}+3\\
=&4\cdot\frac{x_{u_1}}{x_{u^*}}+2\cdot\frac{x_{u_3}}{x_{u^*}}+\frac{x_{u_4}}{x_{u^*}}+\frac{x_{u_5}}{x_{u^*}}+4\cdot\frac{x_{u^{\prime}_1}}{x_{u^*}}+2\cdot\frac{x_{u^{\prime}_3}}{x_{u^*}}+\frac{x_{u^{\prime}_4}}{x_{u^*}}+\frac{x_{u^{\prime}_5}}{x_{u^*}}-18+ \\
&\sum_{H^{\prime\prime} \in \varGamma \setminus \{H,H^{\prime}\}} \eta(H^{\prime\prime})-2\sum_{v \in A_0} \frac{x_v}{x_{u^*}}+3\\
<&-\frac{3}{5},
\end{align*}
a contradiction. Thus, $H^{\prime}\ncong K_5-e$. According to \eqref{equ::3.7}, we get
\begin{align}\label{equ::3.14}
\eta(H)+\eta(H^{\prime})\geq -3.
\end{align}
From Claim \ref{clm::2.1} (i), we know
\begin{align}\label{equ::3.8}
\eta(H)=\eta(K_5-e)\leq -1,
\end{align}
equality holds only if $x_{u_i}=x_{u^*}$ for each $i\in\{1,2,3,4,5\}$.
If $H^{\prime}\cong K_2$, then by \eqref{equ::3.14}-\eqref{equ::3.8} and Claim \ref{clm::2.1}, we find that $\varGamma$ consists of $H$ and $ H^{\prime}$. Thus, $G[A_+]\cong (K_5-e)\cup K_2$, which contradicts $e(A_+) \geq 11$. Hence, $H^{\prime}\ncong K_2$. Notice that $H^{\prime}\ncong K_5$ from Claim \ref{clm::2.2}. Therefore, according to Claim \ref{clm::2.1}, we see that $\eta(H^{\prime})\leq-2$. Combining this with \eqref{equ::3.14} and \eqref{equ::3.8}, we obtain $\eta(H^{\prime})=-2$ and $\eta(H)=-1$. That is, the equality in \eqref{equ::3.8} holds, and so $x_{u_1}=x_{u^*}$, which  contradicts \eqref{equ::3.10}.
\end{proof}		

Note that  $e(A_+) \geq 11$ from Claim \ref{clm::2.4}. Thus, there is at least one component in $\varGamma$ and $|A_+|\geq 6$. By Claims \ref{clm::2.2} and \ref{clm::2.3}, we have $H\notin \{K_5, K_5-e\}$ for each component $H\in \varGamma$. Suppose that there are two components $H, H^{\prime} \in \varGamma$. Then by \eqref{equ::3.7} we obtain $\eta(H)+\eta(H^{\prime})\geq-3$. 
Together with Claim \ref{clm::2.1}, we know that both $H$ and $ H^{\prime}$ must be isomorphic to $ K_2$, and $\varGamma$ consists of $H$ and $ H^{\prime}$. So, $G[A_+]\cong2K_2$, which contradicts $e(A_+) \geq 11$. Hence, $\varGamma$ contains exactly one component, that is, $G[A_+]$ is connected. 

If $G[A_+]\cong K_1 \vee (K_3 \cup (|A_+|-4)K_1)$, then let $A_+=\{u_0,u_1,u_2,u_3,v_1, v_2,\dots,v_{t}\}$, where $d_{A}(u_0)=t+3$, $d_{A}(u_i)=3$ for each $i\in \{1,2,3\}$ and $d_{A}(v_j)=1$ for each $j\in \{1,2,\dots,t\}$. It follows from \eqref{equ::3.4} that
\begin{align*}
0\leq e(B) &\leq \sum_{u \in A_+}(d_{A}(u)-2)\frac{x_u}{x_{u^*}}-e(A_+)-2\sum_{v \in A_0} \frac{x_v}{x_{u^*}}+3\\
&\leq -2- \sum_{j=1}^{t}\frac{x_{v_j}}{x_{u^*}}-2\sum_{v \in A_0} \frac{x_v}{x_{u^*}}+3=1- \sum_{j=1}^{t}\frac{x_{v_j}}{x_{u^*}}-2\sum_{v \in A_0} \frac{x_v}{x_{u^*}},
\end{align*}
which implies 
\begin{align*}
\sum_{j=1}^{t}x_{v_j} +2\sum_{v \in A_0}x_v\leq x_{u^*}.
\end{align*}
Then
\begin{align*}
\rho x_{u^*}=\sum_{i=0}^{3}x_{u_i}+\sum_{j=1}^{t}x_{v_j} +\sum_{v \in A_0} x_v\leq 4x_{u^*}+\sum_{j=1}^{t}x_{v_j} +2\sum_{v \in A_0}x_v\leq4x_{u^*}+x_{u^*}=5x_{u^*},
\end{align*}
and so $\rho \leq 5$, which contradicts \eqref{equ::3.1}. Hence, $G[A_+]\ncong K_1 \vee (K_3 \cup (|A_+|-4)K_1)$. 

If $G[A_+]\cong K_{1,|A_+|-1}$, then let $A_+=\{u_0,u_1,u_2,\dots,u_{s}\}$, where $d_{A}(u_0)=s$ and $d_{A}(u_i)=1$ for each $i\in \{1,2,\dots,s\}$. It follows from  \eqref{equ::3.4} that
\begin{align*}
0\leq e(B) &\leq \sum_{u \in A_+}(d_{A}(u)-2)\frac{x_u}{x_{u^*}}-e(A_+)-2\sum_{v \in A_0} \frac{x_v}{x_{u^*}}+3\\
&\leq -2- \sum_{i=1}^{s}\frac{x_{u_i}}{x_{u^*}}-2\sum_{v \in A_0} \frac{x_v}{x_{u^*}}+3=1- \sum_{i=1}^{s}\frac{x_{u_i}}{x_{u^*}}-2\sum_{v \in A_0} \frac{x_v}{x_{u^*}},
\end{align*}
which indicates
\begin{align*}
\sum_{i=1}^{s}x_{u_i} +2\sum_{v \in A_0} x_v\leq x_{u^*}.
\end{align*}
Then
\begin{align*}
\rho x_{u^*}= x_{u_0}+\sum_{i=1}^{s}x_{u_i} +\sum_{v \in A_0} x_v\leq x_{u^*}+\sum_{i=1}^{s}x_{u_i} +2\sum_{v \in A_0} x_v\leq x_{u^*}+ x_{u^*}=2x_{u^*},
\end{align*}
and hence $\rho \leq 2$, which contradicts \eqref{equ::3.1}. Therefore, $G[A_+]\ncong K_{1,|A_+|-1}$.

Since  $|A_+|\geq 6$, we get $G[A_+]\notin \{K_2, K_4, K_5-2e\}$.
It follows that $G[A_+]\notin \{K_5, K_5-e, K_4, K_5-2e, K_2,K_1 \vee (K_3 \cup (|A_+|-4)K_1), K_{1,|A_+|-1}\}$.  Together with Claim \ref{clm::2.1}, we obtain 
\begin{align}\label{equ::3.11}
\eta(G[A_+])\leq-3, 
\end{align}
equality holds only if $G[A_+] \cong K_2 \vee (|A_+|-2)K_1$ and $x_u=x_{u^*}$ for all $u \in A_+$ with $d_{A_+}(u)\geq3$.
On the other hand,  we see that  $\eta(G[A_+])\geq-3$ by \eqref{equ::3.7}. Thus, $\eta(G[A_+])=-3$, and so the equality in \eqref{equ::3.11} holds. In view of \eqref{equ::3.7}, we get $A_0=\emptyset$ and $e(B)=0$. Consequently, we obtain $G[A] \cong K_2 \vee (|A|-2)K_1.$ 
Suppose that $B\neq\emptyset$. Clearly, there are two vertices $u_1, u_2 \in A$ such that $d_A(u_1)=d_A(u_2) \geq 3$. It follows that $x_{u_1}=x_{u_2}=x_{u^*}$.  Let $w$ be an arbitrary vertex in $B$. If $N_G(w)\cap \{u_1,u_2\}\neq\emptyset$, then we may assume without loss of generality that $N_B(u_1)\neq\emptyset$, which indicates
\begin{align*}
\rho x_{u_1}=x_{u^*}+\sum_{u\in A\setminus \{u_1\}}x_{u}+\sum_{w^{\prime}\in N_B(u_1)}x_{w^{\prime}}=\sum_{u\in A }x_{u}+\sum_{w^{\prime}\in N_B(u_1)}x_{w^{\prime}}>\sum_{u\in A}x_{u}=\rho x_{u^*},
\end{align*}
and so $x_{u_1}>x_{u^*}$, a contradiction. Hence, $N_G(w)\cap \{u_1,u_2\}=\emptyset$. Recall that $G$ has no isolated vertices and $e(B)=0$. Then $d_A(w) \geq 1$ and $N_G(w)\subseteq A \setminus \{u_1,u_2\} \subsetneq N_G(u^*)$. Thus we have
$$
\rho x_{w}=\sum_{v\in N_G(w)}x_{v}<\sum_{u\in N_G(u^*)}x_{u}=\rho x_{u^*},
$$
which gives $x_w< x_{u^*}$. Notice that the equality in \eqref{equ::3.4} holds. Hence, $x_{w}=x_{u^{*}}$, a contradiction. 
Therefore, $B=\emptyset$, i.e., $G\cong K_{3} \vee (|A|-2)K_1$. Since $G$ is a graph with size $m$,  we have $|A|-2=\frac{m-3}{3}$. That is, $G\cong K_3 \vee \frac{m-3}{3}K_1$, as desired. This completes the proof.
\end{proof}

\section*{Declaration of competing interest}
The authors declare that they have no conflict of interest.

\section*{Data availability}
No data was used for the research described in the article.


\begin{thebibliography}{99}
\setlength{\itemsep}{0pt}

\bibitem{BBR}
R. Bapat, Graphs and Matrices, Springer, New York, 2010.

\bibitem{BAJ}
J.A. Bondy, U.S.R. Murty, Graph Theory, Springer, New York, 2008.

\bibitem{Brou}
A.E. Brouwer, W.H. Haemers, Spectra of Graphs, Springer, Berlin, 2011.

\bibitem{BRU}
R.A. Brualdi, A.J. Hoffman, On the spectral radius of $(0,1)$-matrices, \emph{Linear Algebra Appl.} 65 (1985) 133--146.

\bibitem{Gods}
C. Godsil, G. Royle, Algebraic Graph Theory, in: Graduate Texts in Mathematics, vol. 207, Springer-Verlag, New York, 2001.

\bibitem{Li1}
Y. Li, L. Lu, Y. Peng, Spectral extremal graphs for the bowtie, \emph{Discrete Math.} 346 (2023) 113680.

\bibitem{Lu}
J. Lu, L. Lu, Y. Li, Spectral radius of graphs forbidden $C_7$ or $C_6^{\triangle}$, \emph{Discrete Math.} 347 (2024) 113781.

\bibitem{Min}
G. Min, Z. Lou, Q. Huang, A sharp upper bound on the spectral radius of $C_5$-free/ $C_6$-free graphs with given size, \emph{Linear Algebra Appl.} 640 (2022) 162--178.

\bibitem{NVB}
V. Nikiforov, Some inequalities for the largest eigenvalue of a graph, \emph{Combin. Probab. Comput.} 11 (2002) 179--189.

\bibitem{NVB1}
V. Nikiforov, Walks and the spectral radius of graphs, \emph{Linear Algebra Appl.} 418 (2006) 257--268.

\bibitem{NVBF}
V. Nikiforov, The maximum spectral radius of $C_4$-free graphs of given order and size, \emph{Linear Algebra Appl.} 430 (2009) 2898--2905.

\bibitem{NVBF3}
V. Nikiforov, More spectral bounds on the clique and independence numbers, \emph{J. Combin. Theory Ser. B} 99 (2009) 819--826.

\bibitem{NVB2}
V. Nikiforov, On a theorem of Nosal, arXiv:2104.12171, 2021.

\bibitem{NEE}
E. Nosal, Eigenvalues of Graphs (Master's thesis), University of Calgary, 1970.

\bibitem{Row}
P. Rowlinson, On the maximal index of graphs with a prescribed number of edges, \emph{Linear Algebra Appl.} 110 (1988) 43--53.


\bibitem{Wu}
B. Wu, E. Xiao, Y. Hong, The spectral radius of trees on $k$ pendant vertices, \emph{Linear Algebra Appl.} 395 (2005) 343--349.

\bibitem{Yu}
L. Yu, Y. Li, Y. Peng, Spectral extremal graphs for fan graphs, arXiv:2404.03423, 2024.

\bibitem{MZS}
M. Zhai, H. Lin, J. Shu, Spectral extrema of graphs with fixed size: cycles and complete bipartite graphs, \emph{European J. Combin.} 95 (2021) 103322.

\bibitem{Zhang}
Y. Zhang, L. Wang, On the spectral radius of graphs without a gem, manuscript submitted, March 19, 2024.	
\end{thebibliography}
\end{document}